\documentclass[12pt]{article}
\usepackage{amssymb,amsmath}

\usepackage[latin1]{inputenc}
\usepackage[T1]{fontenc}
\usepackage{play}

\usepackage{graphicx}
\DeclareGraphicsExtensions{.pdf,.png,.jpg}

\graphicspath{ {images/} }

\usepackage{amsmath}
\usepackage{amsthm}
\usepackage{amsfonts}
\usepackage{bbm}
\usepackage{amssymb}
\usepackage{graphicx}
\usepackage{color}
\usepackage{subcaption}

\usepackage[dvipsnames]{xcolor}

\usepackage{fancybox}
\newlength{\querylen}
\newcommand{\sgn}{\mathop{\mathrm{sgn}}}

\newcommand{\prob}{\mathbb{P}}

\newcommand{\me}{\mathbb{E}}
\newcommand{\E}{\mathbb{E}}

\newcommand{\tX}{\widetilde{X}}
\newcommand{\tL}{\widetilde{L}}

\newcommand{\tpsi}{\widetilde{\psi}}

\newtheorem{thm}{Theorem}
\newtheorem{lemma}[thm]{Lemma}

\theoremstyle{definition}

\theoremstyle{remark}

\makeatletter
\renewcommand*\env@matrix[1][\arraystretch]{%
  \edef\arraystretch{#1}%
  \hskip -\arraycolsep
  \let\@ifnextchar\new@ifnextchar
  \array{*\c@MaxMatrixCols c}}
\makeatother

\begin{document}
\title{Diffusion Approximations in the Online Increasing Subsequence Problem} 
\author{Alexander Gnedin and Amirlan Seksenbayev}

\maketitle
\noindent

\begin{abstract}
\noindent
The online increasing subsequence problem is a stochastic optimisation task with the objective to maximise the expected
length of subsequence chosen from a random series by means of a nonanticipating decision strategy.
We study the structure of optimal and near-optimal subsequences in a standardised planar Poisson framework.
Following a long-standing suggestion by Bruss and Delbaen \cite{BD2}, we prove
 a joint functional limit theorem for the transversal fluctuations about the diagonal of the running maximum and the length processes.
The limit is identified explicitly with a Gaussian time-inhomogeneous diffusion. In particular, the running maximum converges to a Brownian bridge, 
and the length process has   another explicit  non-Markovian limit.

\end{abstract}

\section{Introduction}
We adopt the following standardised framework for the {\it online} stochastic optimisation problem of Samuels and Steele \cite{SS}.
Suppose i.i.d. marks drawn from the uniform distribution on $[0,1]$ are observed, one by one, 
   at times of independent  homogeneous Poisson process of intensity $\nu$ on $[0,1]$. Each mark can be selected or rejected.
The sequence of selected marks must increase. The task is to maximise  the expected length of selected increasing subsequence  using   an online   strategy.  
The online constraint requires that
each decision  becomes immediately  irrevocable as the mark is observed and
must be based exclusively on the information accumulated  previously without  foresight of the future.

The optimal online strategy is defined recursively in terms of a variable acceptance window, which limits  the difference between the next  and previous selections.
The strategy and its value can be found, in principle, by solving a dynamic programming equation, see
  \cite{BG, BD1, GS} for properties of the solution and approximations.
Our main theme is different.
We study the time evolution of increasing subsequences  under  online strategies that are  
 within $O(1)$  gap from the optimum for large $\nu$.

Let  $L_\nu(t)$ and $X_\nu(t)$, respectively, denote the length  and the last element of the increasing subsequence selected by time $t\in[0,1]$ under the optimal  strategy. 
The interest  to date focused on the total length 
$L_\nu(1)$. Samuels and Steele derived the principal asymptotics  ${\mathbb E}L_\nu(1)\sim\sqrt{2\nu}$, which was later found to be  an upper bound
with the optimality  gap of   order $\log \nu$ \cite{BD1}.  See \cite{Arlotto1, Arlotto 2, BG, BD1,  G-random} for refinements and generalisations. 
 In a recent paper \cite{GS} we combined
 asymptotic analysis of the dynamic programming equation 
with a renewal approximation to the range of the process $Z_\nu(t):=\sqrt{\nu(1-t)(1-X_\nu(t))}$ 
to  derive expansions for the mean
\begin{equation}\label{Eopt}
{\mathbb E}L_\nu(1)\sim\sqrt{2\nu}-\frac{1}{12}\log  \nu+c_0^*, ~~~\nu\to \infty,
\end{equation}
and the variance
\begin{equation}\label{Vopt}
{\rm Var }\,L_\nu(1)\sim\frac{1}{3}\sqrt{2\nu}-\frac{1}{72}\log \nu+c_1^*, ~~~\nu\to \infty,
\end{equation}
where $c_0^*$ and $c_1^*$ are unknown constants.
A central limit theorem for $L_\nu(1)$ was proved in \cite{BD2} 
by analysis of a related martingale,
and further extended in \cite{GS} to  a larger class of asymptotically optimal strategies
by the mentioned  renewal theory approach.

Samuels and Steele introduced the online selection problem as 
 an  offspring of  the much-studied Ulam-Hammersley problem
on the longest increasing subsequence of the Poisson scatter of points in the square $[0,1]^2$.  They viewed the Ulam-Hammersley problem as an  {\it offline} optimisation task with the complete foresight of choosable options. 
In the offline framework,  
 the well known principal asymptotics of the expected maximum length, $2\,\sqrt{\nu}$,
 is similar, but 
the second term of its asymptotic  expansion  and the principal term of  the standard deviation are both of the order  $\nu^{1/6}$.
The limit law for the offline maximum  length is the Tracy-Widom distribution from the random matrix theory. 
An excellent reference here with an exposition of the history is the book by Romik \cite{R}.

In the offline problem, 
some work has been done on the size of 
  transversal fluctuations about the diagonal $x=t$ in $[0,1]^2$.  
Johansson \cite{J} proved a measure concentration result asserting that, with probability approaching 1,  every longest increasing subsequence (which is not unique) lies in a diagonal strip of width of the order $\nu^{-1/6+\epsilon}$.
Duvergne, Nica and Vir{\'a}g \cite{DNV} recently proved the existence and gave some description of the functional limit, which is not Gaussian.
But for
smaller exponent $-1/2<\alpha<- 1/6 $, Joseph and Peled  \cite{DJP} showed that if  the increasing sequence is restricted to lie within the strip of width  $\nu^{-\alpha}$,   the expected maximum length  remains to be   asymptotic to $2\sqrt{\nu}$, while the limit distribution of the length switches to normal.

To extend the parallels  and gain further insight into the optimal selection  it is of considerable interest to  examine fluctuations of the  processes $L_\nu$ and $X_\nu$ as a whole.
On this path, one is  lead to study the  following scaled and  centred versions of the running maximum and length processes:
\begin{equation}\label{scale}
\widetilde{X}_\nu(t):=\nu^{1/4}(X_\nu(t)-t),~~~~~~ \widetilde{L}_\nu(t)= \nu^{1/4}\left(   \frac{ L_\nu(t)}{\sqrt{2\nu}}-t\right), ~~~t\in[0,1].
\end{equation}
To compare, in  the offline problem  by similar centring    the critical  transversal and longitudinal scaling factors  appear to be $\nu^{1/6}$ and $\nu^{1/3}$, respectively.
Our central result (Theorem \ref{main res})  is a functional limit theorem which entails that the process
$(\widetilde{X}_\nu,\widetilde{L}_\nu)$
converges weakly to a simple two-dimensional Gaussian diffusion.
In particular, $\widetilde{X}_\nu$ approaches a Brownian bridge.
The limit of $\widetilde{L}_\nu$ is a non-Markovian process 
with the covariance function 
$$  (s,t)\mapsto    \frac{2s(2-t) -(2-s-t)\log(1-s)}{6\sqrt{2}}\,,~~~~0\leq s\leq t\leq 1,$$
which corresponds to a correlated  sum of a Brownian motion and a Brownian bridge,

The question about functional limits for $L_\nu$ and $X_\nu$  has been initiated by Bruss and Delbaen  
\cite{BD2}.
 They  employed  the Doob-Meyer decomposition to compensate the processes, and in an analytic tour de force
showed that  the scaled martingales jointly converge  to a correlated Brownian motion in two dimensions.
However, the compensation
  keeps out of sight  a drift component absorbing much of the fluctuations  immanent to the selection process, let alone that 
the compensators themselves are nonlinear integral transforms of $X_\nu$.
To break the vicious circle one needs to obtain the limit of $X_\nu$  under complete control over the centring. 
Curiously, in the forerunning  paper  Bruss and Delbaen mentioned that  P.A. Meyer had suggested to them 
to  scrutinise  the generator of  the Markov process $(X_\nu,L_\nu)$    (see \cite{BD1}, Remark 2.4).

Looking at the generator of (\ref{scale}) we shall recognise  the limit process without difficulty.
But in order to justify the weak convergence in the Skorokhod space on the closed interval $[0,1]$
we will need to circumvent a difficulty  caused by pole singularities of the control function and the drift  coefficient at the right endpoint.
We shall also discuss related processes and derive tight uniform bounds on the expected values of $X_\nu$ and $L_\nu$, thus embedding 
the moment expansions from \cite{GS} in the functional context.

\vskip0.2cm
{\it Notation.} We sometimes omit  dependence on the intensity parameter $\nu$ wherever there is no ambiguity. Notation $X$ and $L$ will be context-dependent, typically standing for processes associated with a near-optimal online selection strategy,
while  $\tX$ and $\tL$ will denote the normalised versions with scaling and centring  as in (\ref{scale}).

\section{Selection strategies}
It will be convenient 
to extend the underlying framework slightly by considering
 a homogeneous  Poisson random measure  $\Pi$  with  intensity $\nu$ in the halfplane ${\mathbb R}_+\times {\mathbb R}$, along 
 with the  filtration induced by  restricting $\Pi$ to $[0,t]\times{\mathbb R}$ for  $t\geq 0$.
We interpret  the generic atom $(t,x)$ of $\Pi$ as  random mark $x$ observed at time $t$. 
A sequence $(t_1,x_1), \dots, (t_\ell,x_\ell)$  of atoms   is said to be increasing if it is a chain in two dimensions, i.e.
$$t_1<\dots<t_\ell,~x_1<\dots<x_\ell.$$

For a given bounded measurable {\it control} function $\psi:[0,1]\times{\mathbb R}\to {\mathbb R}_+$, an online  strategy selecting such increasing sequence is defined by the following intuitive rule.
Let $x$ be the last mark selected before time $t$, or some given $x_0$ if no selection has been made.
Given the next mark $x'$ is observed at time $t$, this mark  is selected if and only if 
\begin{equation*}
   x<x'\leq x+\psi(t,x).
\end{equation*}
One can think of more general
online strategies, with the acceptance window shaped differently from  an interval  or possibly depending on the history in a more complex way.
Yet the considered class is sufficient for the sake of optimisation and can be further reduced to controls of a special type.

For a given control $\psi$, define  $X(t)$ and $L(t)$ to be, respectively,  the last mark selected and the number of marks selected within the time interval $[0,t]$.
The process  $X=(X(t),~t\in[0,1])$, which we call the {\it  running maximum}, is a time-inhomogeneous Markov process,
 jumping from the generic  state $x$ at rate $\psi(t,x)$ to another state uniformly distributed on $[x, x+\psi(t,x)]$.
The {\it length process}  $L=(L(t),~t\in[0,1])$ just counts the jumps of $X$; hence the bivariate process
$(X,L)$ is also Markovian.
Moreover,    the conditional distribution of $((X(t),L(t)), ~t\geq s)$ depends on the  pre-$s$ history only through $X(s)$.

Intuitively, the bigger $\psi$, the faster $X$ and $L$ increase. To enable 
comparisons of 
  selection processes with different controls it is very convenient to couple them by means of an 
additive representation through another Poisson random measure ${\Pi}^*$, thought of as a reserve of  positive increments.
The underlying properties of the planar Poisson process are  translation invariance and spatial independence:
$\Pi$ restricted to the shifted quadrant $(t,x)+{\mathbb R}_+^2$ is independent of $\Pi_{|[0,t]\times{\mathbb R}}$ and   has the same distribution as the translation of $\Pi_{|{\mathbb R}_+^2}$ by vector $(t,x)$.
So, letting ${\Pi}^*$ to be a distributional copy of $\Pi$, a solution to the system of stochastic differential equations
\begin{eqnarray}\label{knaps}
{\rm d}X(t)=\int_0^{\psi(t,X(t))} x \,{\Pi}^*({\rm d}t{\rm d}x), ~~~
{\rm d}L(t)=\int_0^{\psi(t,X(t))} {\Pi}^*({\rm d}t{\rm d}x)
\end{eqnarray}
with initial values $X(0)=x_0$ and $L(0)=0$
 will have the same distribution as $(X,L)$.

By virtue  of the  additive realisation through $\Pi^*$, the online increasing subsequence problem is transformed into an online knapsack packing problem \cite{Binpacking}.  
Here, the generic item of some size $x$
observed at time $t$ (an  atom  of $\Pi^*$)
can be either packed or dismissed.   The objective  translates as  maximisation of the expected number of items added within the unit time horizon to a knapsack of  unit capacity.
Note that for the increasing subsequence problem the (continuous) distribution of marks does not matter, while the knapsack problem is not distribution-free.

\begin{lemma}\label{compare}  For $i=1,2$ 
let $X_i$ be selection processes driven by  controls $\psi_i$. 
By coupling via {\rm (\ref{knaps})}, each time a process with smaller acceptance window jumps, the other process also has a jump of the same size.
\end{lemma}
\begin{proof}
Straight from (\ref{knaps}),
\begin{eqnarray*}
{\rm d}(X_1-X_2) &=& \sgn(\psi_1-\psi_2) 
 \int_{\psi_1 \wedge\psi_2}^{\psi_1 \vee \psi_2}  x \,{\Pi}^*({\rm d}t{\rm d}x),\\
\end{eqnarray*}
where for shorthand $\psi_i=\psi_i(t,X_i(t))$.
\end{proof}

Conditionally on $(X(s), L(s))=(x,\ell)$, the process 
$(X(s+\cdot)-x, L(s+\cdot)-\ell))$ has the same distribution as $(X^{(s,x)}, L^{(s,x)})$, which similarly to (\ref{knaps}) is given by 
$$dX^{(s,x)}(u)=\int_0^{\psi(s+u,x+X^{(s,x)}(u))} y\Pi^*({\rm d}u{\rm d}y),~~~{\rm d}L^{(s,x)}(u)=\int_0^{\psi(s+u,x+X^{(s,x)}(u))} \Pi^*({\rm d}u{\rm d}y).$$
Averaging, we obtain  formulas for  the predictable {\it compensators} of $X$ and $L$
\begin{equation}\label{CXCL}
C_X(t):= \frac{\nu}{2} \int_0^t \psi^2(s, X(s)){\rm d}s,~~~C_L(t):= \nu \int_0^t \psi(s, X(s)){\rm d}s,
\end{equation}
so $X-C_X, L-C_L$ are zero-mean martingales.

With every control we may further relate a zero-mean martingale
\begin{equation}\label{Mmart}
M(t):=L(t)+\E\{L(1)-L(t)|X(t)\}-\E L(1)
\end{equation}
with terminal value $L(1)-\E L(1)$. If $\psi$ does not depend on $x$,  $L$ has independent increments and $M(t)= L(t)-\E L(t)$.

The selected increasing chain fits in the unit square if $X(1)\leq 1$, which translates in terms of the control function as the condition of 
{\it feasibility}:
\begin{equation*}
0<\psi(t,x)\leq 1-x~~{\rm for~~}(t,x)\in [0,1]^2.
\end{equation*} 
In the sequel, if not stated otherwise  we set $x_0=0$ and only consider feasible controls.

\subsection{Principal convergence of the moments}

This section follows closely the arguments found in \cite{BD2}, pp. 291-292.

Let 
$$p(t):=\E X(t)=\E C_X(t),~~~q(t):=\frac{\E L(t)}{\sqrt{2\nu}}=\frac{\E C_L(t)}{\sqrt{2\nu}}.$$
Some general relations between the moments follow straight from formulas for the compensators (\ref{CXCL}).  For shorthand, write $\psi=\psi(X(s),s)$. We have
$$0\leq \E\int_0^t\left( 1\pm\sqrt{\nu/2}\,\psi \right)^2{\rm d}s=t\pm 2q(t)+p(t),$$
where the right-hand side in increasing in $t$. It follows,
\begin{equation}\label{pq}
p(t)-t\geq 2\,(q(t)-t).
\end{equation}
Using the Cauchy-Schwarz inequality
\begin{eqnarray}\label{pq1}
\nonumber
(p(t)-t)^2= \left( \E\int_0^t \left(1-{\frac{\nu}{2}}\psi^2\right){\rm d}s\right)^2\leq \\
\nonumber
 \E\int_0^t \left(1+\sqrt{\nu/2}\,\psi\right)^2{\rm d}s\,\,
 \E\int_0^t \left(1-\sqrt{\nu/2}\,\psi\right)^2{\rm d}s=\\
(t+2q(t)+p(t))(t-2q(t)+p(t)).
\end{eqnarray}
Similarly
\begin{eqnarray}\label{pq2}
(q(t)-t)^2=  \left( \E\int_0^t 1\cdot\left(1-\sqrt{\nu/2} \,\psi \right){\rm d}s\right)^2\leq 
t\,(t-2q(t)+p(t))
\end{eqnarray}

The above relations did not use the feasibility constraint.
For feasible control we have $p(1)<1$, hence from (\ref{pq}) also $q(1)<1$. Since all factors in the right-hand sides  of 
 (\ref{pq1}), (\ref{pq2}) are increasing, replacing them by their maximal values at $t=1$ we obtain
\begin{equation}\label{pq3}
(p(t)-t)^2<8(1-q(1)),~~~(q(t)-t)^2<2(1-q(1)).
\end{equation}

We say that a strategy $\psi=\psi_\nu$ is {\it asymptotically optimal in the principal term} if $q(1)\to 1$, as $\nu\to\infty$, i.e. $\E L_\nu(1)\sim\sqrt{2\nu}$; in that case (\ref{pq3}) imply
the uniform convergence of the moments
$$\sup_{t\in[0,1]}|p(t)-t|\to0, \sup_{t\in[0,1]}|q(t)-t|\to0.$$

It follows from (\ref{Eopt}) that under the optimal strategy
\begin{equation}\label{q1asymp}
1-q(1)\sim \frac{\log \nu}{12\sqrt{2\nu}},~~~\nu\to\infty.
\end{equation}
This relation can be called  a {\it two-term asymptotic optimality}. 
Whenever this holds,
 the general bounds (\ref{pq3}) imply that both $\sup_{t\in[0,1]}|p(t)-t|$ and $\sup_{t\in[0,1]}|q(t)-t|$ can be estimated as $O(\sqrt{\log \nu}/\nu^{1/4})$.
A refinement of the convergence rate will be obtained in Section \ref{Moments}.

\subsection{The greedy strategy}
The {\it greedy} strategy, with control 
$\psi(t,x)=1-x$, outputs the sequence of consecutive records. The strategy is optimal for $\nu<1.34...$
Statistical properties of records from the Poisson process is a much-studied subject
 \cite{Records}. It is well known that, as $\nu$ increases, the distribution of $L(1)$ approaches normal
with mean and  variance  both asymptotic to $\log \nu$. Normalisation (\ref{scale}) is not 
 appropriate here as most of the records concentrate near the north-west corner of the unit square.

\subsection{The stationary strategy}\label{stationary}
We call the strategy with control $\psi(t,x)=\sqrt{{2}/{\nu}}$ {\it stationary}. 
Although not  feasible, the stationary strategy is an important benchmark.
Clearly, $L$ is a Poisson counting process with intensity
${\mathbb E}L(1)=\sqrt{2\nu}$.
Taking general constant  control $\psi(t,x)=\sqrt{c/\nu}$
with some  $c>0$  will yield a strategy outputting the mean length $\sqrt{\{c\wedge (2/c)\} \nu}$, which is maximal for $c={2}$.
In fact,
a much stronger optimality property holds:  the stationary strategy 
achieves the maximum
 expected length over the class of strategies that 
satisfy the {\it mean-value} constraint ${\mathbb E}X(1)\leq1$, see 
\cite{Arlotto1, BG, G-random,  GS}
for  proof and generalisations.
This gives the well-known upper bound mentioned in the Introduction, because each feasible strategy meets the mean-value constraint.

 It is seen from (\ref{knaps}) that $X$ is a compound Poisson process
$$X(t)=\sqrt{\frac{2}{\nu}} \,\,\,\sum_{i=0}^{L(t)} U_i\,,$$
where
$U_1,U_2,\dots$ are independent of $L$,  uniformly distributed  on $[0,1]$.
Straightforward calculation of moments using  Wald's identities yields
$${\mathbb E} X(t)=t,~~~{\rm Var} X(t)= \frac{2^{3/2}t}{3\sqrt{\nu}},~~~{\rm Cov}(X(t),L(t))= t .$$

Since $(X,L)$ has independent increments,
a functional limit 
in the Skorohod topology on $D[0,1]$
follows easily  from the multidimensional   invariance principle: 
\begin{equation*}
(\tX,\widetilde{L})\Rightarrow (W_1,  W_2),~~{\rm as~}\nu\to\infty,
\end{equation*}
where $\Rightarrow$ denotes   weak convergence,  and the limit process
 ${\boldsymbol W} :=(W_1,W_2)$ is a  two-dimensional Brownian motion with zero drift and covariance matrix
\begin{equation}\label{Sigma}
\me  \{ {\boldsymbol W}(t)^T \boldsymbol{W}(t)\}=t\,{\boldsymbol\Sigma}, ~~{\rm where}~~{\boldsymbol\Sigma}:=
\large
{
\begin{pmatrix}
\frac{2\sqrt{2}}{3} &\frac{1}{\sqrt{2}}\\ 
   \frac{1}{\sqrt{2}} & \frac{1}{\sqrt{2}}
\end{pmatrix}
}
\end{equation}
So, marginally,  $W_1$ and $W_2$ are centred Brownian motions with diffusion coefficients and correlation, respectively,

\begin{equation}\label{sigmas}
        \sigma_1:=\frac{2^{3/4}}{\sqrt{3}},~~~\sigma_2:=\frac{1}{2^{1/4}}, ~~~\rho:= \frac{\sqrt{3}}{2}. 
\end{equation}

Notably,
$$\rho=\frac{\sigma_2}{\sigma_1},$$
which implies that
the process $\boldsymbol W$ satisfies the identity
\begin{equation}\label{IdentW}
2W_2-W_1\stackrel{d}{=} W_1,
\end{equation}
which has a pre-limit analogue
\begin{equation}\label{2L-X}
2\tL-\tX  \stackrel{d}{=} \tX.
\end{equation}

Identity (\ref{2L-X}) can be explained by the symmetry
 of the uniform distribution, $U_i\stackrel{d}{=}1-U_i$, which allows us to write
$$X(t)\stackrel{d}{=} \sqrt{\frac{2}{\nu}} \,\,\,\sum_{i=0}^{L(t)} (1-U_i)=  \sqrt{\frac{2}{\nu}} \,  L(t)-X(t)\,.$$

Martingale (\ref{Mmart}) just coincides with naturally centred $L$.

The correlated Brownian motion has  appeared in Bruss and Delbaen \cite{BD2} (Theorem 4.1), as the limit of
 $(X, L)$  centred by their compensators $C_{X}$ and $C_{L}$
under the optimal  (feasible) strategy. 
This connection confirms that the key to the fluctuation problem is 
understanding the nature of the drift component.

\subsection{A feasible version of the stationary strategy}\label{limit-jump}\label{fstationary}

 The strategy driven by   $\psi(t,x)=\sqrt{2/\nu}\,\wedge (1-x)$ 
 is a counterpart of that introduced  by Samuels and Steele in the discrete-time setting \cite{SS}. 
This is a minor modification of the stationary strategy to meet the feasibility condition.
Define the hitting time
$$\tau:=\inf \{t\in [0,1]: X(t)\geq 1-\sqrt{2/\nu}\}$$
with the convention $\inf\varnothing=1$.
The strategy acts as the stationary before $\tau$, and  if $\tau<1$  proceeds with a greedy selection, so, in essence, the selection process is frozen at  time $\tau$.
Using elementary renewal theory arguments, we find asymptotics of the moments
$${\mathbb E}L(1)\sim \sqrt{2\nu}-\frac{2^{3/4}}{\sqrt{3\pi}}\,\nu^{1/4},~~~{\rm Var} L(1)\sim \frac{2^{3/2}}{3}\left( 1-\frac{1}{\pi}\right)\sqrt{\nu}.$$
Hence the strategy is asymptotically optimal in the principal term.

Furthermore, $\widetilde{L}(1)$ converges in distribution to 
${2^{-1/4}}\{ {(\xi_1/\sqrt{3})\wedge \xi_2}\},$
where $\xi_1,\xi_2$ are independent ${\cal N}(0,1)$, see \cite{GS}. The normalised terminal value  $\widetilde{X}(1)$ is nonpositive, and converges in distribution to $-2^{3/4}3^{-1/2}(\eta)_+$, where 
$\eta$ is another standard normal variable and $(\cdot)_+$ denotes the positive part. 
By symmetry of the normal distribution, the hitting time $\tau$ 
assumes value $1$ with probability approaching $\prob(V_1(1)<0)=1/2$, 
and otherwise  $1-\tau$ is of the order $\nu^{-1/4}$.  Comparing with the stationary strategy, one can see that there is an optimality gap of order $\nu^{1/4}$ which occurs due to a premature 
freeze of selection in the event $\tau<1$.

Note that the moments of terminal values satisfy $1-p(1)\sim c_1\nu^{-1/4}, 1-q(1)\sim c_2\nu^{-1/4}$ with some $c_1, c_2>0$, while (\ref{pq3}) overestimates the first as
$1-p(1)=O(\nu^{-1/8})$.

In terms of the normalised running maximum, $\tau$ is the time when $\tX$ hits the straight line connecting points $(0,\nu^{1/4})$ and $(1,0)$.
Since  $\tau\to 1$ in probability,
 $(\tX,\tL)$ has the same functional limit as  under the stationary strategy  on every interval $[0,1-h]$, for $h\in (0,1)$. 
Extending the functional limit to the closed $[0,1]$ leads to a discontinuity at $t=1$. To capture the jump, it is enough to modify the correlated Brownian motion $\boldsymbol W$ 
by  replacing the terminal value $(W_1(1), W_2(1))$    with 
$$(W_1(1)-(W_1(1))_+, W_2(1)-(2/\sqrt{3})(W_1)_+).$$

\subsection{Self-similar asymptotically optimal strategies}

We call strategy  {\it self-similar}  if the control $\psi=\psi_\nu$ has the form
\begin{equation}\label{ss-psi}
\psi (t,x):=(1-x)\delta(\nu(1-t)(1-x)), ~~~(t,x)\in [0,1]^2.
\end{equation}
 for some function 
$\delta:{\mathbb R}_+\to [0,1]$. Note that such a strategy is feasible and $\psi_\nu(0,0)=\delta(\nu)$.
The rationale behind this definition is the following. Assuming $x$ to be the running maximum at time $t$, 
the remaining part of the chain should be selected from the north-east rectangle spanned on $(t,x)$ and $(1,1)$, and  by the optimality principle the subsequence selected from the rectangle should have  maximal 
expected length. Mapping the rectangle onto $[0,1]^2$ it is readily seen that the subproblem 
is an independent replica of  the original problem of optimal selection from the unit square
with  intensity parameter  $\nu(1-x)(1-t)$.  
The martingale (\ref{Mmart}) assumes the form 

\begin{equation}\label{Mmart1}
M(t)=L(t)+F(\nu(1-t)(1-X(t))-F(\nu),
\end{equation}
where the {\it value function} $F$, for given control,  depends  on one variable
$$F(\nu):=\E L_\nu(1),~~F(0)=0.$$

\vskip0.2cm
\noindent
{\bf Assumption.}  From this point on we assume that {\it the selection strategy is self-similar  as defined by {\rm (\ref{ss-psi})}, with function $\delta$ having asymptotics 
\begin{equation}\label{delta-asymp}
\delta(\nu)=\sqrt{{2}/{\nu}} +O\left({\nu}^{-1}\right), ~~~\nu\to\infty.
\end{equation}
}
\vskip0.2cm

The assumption  is central and deserves  comments. Whenever  $\nu(1-x)(1-t)$ is large, (\ref{delta-asymp})  implies asymptotics of the control
\begin{equation}\label{psi-appr}
\psi(t,x)\sim \sqrt{\frac{2(1-x)}{\nu(1-t)}},
\end{equation}
which shows that 
near the diagonal $x=t$ the acceptance window is about the same as for the stationary strategy. Away from the diagonal, the acceptance window is close to that for the stationary strategy
adjusted to the rectangle north-east of $(t,x)$.

It is known \cite{GS} that the optimal strategy satisfies  the asymptotic expansion  
\begin{equation}\label{delta-opt}
\delta^*(\nu)\sim \sqrt{2/\nu}-(3\nu)^{-1}+O(\nu^{-3/2}).
\end{equation}
A minor   adjustment of Theorem 6 in  \cite{GS} shows that if we assume, more generally,   the relation $\delta(\nu)\sim \sqrt{2/\nu}+\beta/\nu$ 
with some parameter $\beta\in{\mathbb R}$,
then  asymptotic expansions of the moments (\ref{Eopt}), (\ref{Vopt}) are still valid, with only constant terms depending on $\beta$.
Using a sandwich argument based on  Lemma \ref{compare}, one can further show that 
under the assumption  (\ref{delta-asymp}) expansions of the moments hold but with constant terms being replaced by some $O(1)$ remainders. 
In particular, condition (\ref{delta-asymp}) ensures the two-term asymptotic optimality (\ref{q1asymp}), equivalent 
to the  asymptotic expansion of the value function,
\begin{equation}\label{exF}
F(\nu)=\sqrt{2\nu}-\frac{1}{12}\log(  \nu+1)+O(1).
\end{equation}
We stress that 
the logarithmic term here (as well as in the counterpart of the variance formula (\ref{Vopt})) is not affected by the remainder in (\ref{delta-asymp}), rather 
appears due to the self-similar adjustment of (a feasible version of) the stationary strategy, as
incorporated in (\ref{psi-appr}). 
The impact of the second term in (\ref{delta-asymp}) on moments of the running maximum will be scrutinised in Section \ref{Moments}.

Approximation (\ref{psi-appr}) is not useful  when $t$ or $x$ are too close to  $1$, so that $\nu(1-t)(1-x)$ varies within $O(1)$.
To embrace the full range of the variables, for the sequel we choose $\beta>1$ large enough to meet the bounds 
\begin{eqnarray}\label{small x}
\left|\psi(t,x)-\sqrt{\frac{2(1-x)}{\nu(1-t)}}\right|&<&\frac{\beta}{\nu(1-t)},~~~{\rm for}~~(t,x)\in[0,1)\times[0,1).
\end{eqnarray}
This will be employed along with the  bound
\begin{eqnarray}
\label{big x}
\psi (t,x)&<&\frac{1}{\nu(1-t)},~~~{\rm for}~~ 1-x< \frac{1}{\nu(1-t)}
\end{eqnarray}
which follows by feasibility.

\section{Generators}

The selection process in Section \ref{limit-jump} demonstrates  one type of possible pathology,  caused by large overshooting the diagonal at times close to $t=1$.
But under (\ref{delta-asymp})  it is not even obvious that  $(\tX,\tL)$  has a sensible limit in $D[0,1]$.
A major technical difficulty in showing the convergence is the singularity of (\ref{psi-appr}) at $t=1$. This will be handled in two steps. 
First, we  bound the time variable away from $t=1$ and show the convergence of the generators on a sufficiently big space of test functions. Then we will apply domination arguments 
to bound fluctuations near the right endpoint, thus
justifying  convergence on the full $[0,1]$.

The processes we consider are not time-homogeneous, therefore by computing generators we include the time variable in the state vector.
From (\ref{knaps}), the generator of the jump process $(X,L)$ is
$${\cal L}_\nu f(t,x,\ell)=f_t(t,x,\ell)+\nu\int_0^{\psi(t,x)}\{f(t,x+u,\ell+1)-f(t,x,\ell)\}{\rm d}u.$$
For the processes centered by $t$ we should include $-f_x-f_\ell$ in the generator. Then, with the change of variables
$$x\to x\nu^{-1/4}+t, ~\ell\to (\ell\nu^{-1/4}+t)\sqrt{2\nu},~~\widetilde{\psi}(t,x):=\nu^{1/4} \psi(t, x\nu^{-1/4}+t)$$
we arrive at the generator of $(\tX,\tL)$
\begin{equation}\label{genXL}
\widetilde{\cal L}_\nu f=f_t -\nu^{-1/4} (f_x+f_\ell)+\nu^{3/4}\int_0^{\widetilde{\psi}(t,x)} \{f(t,x+u,\ell+v)-f(t,x,\ell)\}{\rm d}u, 
\end{equation}
where we abbreviate  $f=f(t,x,\ell)$ etc., and
\begin{equation}\label{vi}
v:=(4\nu)^{-1/4}
\end{equation}

We extend $\widetilde{\cal L}_\nu f$  by $0$ outside   the reachable range of $(\tX,\tL)$.
Note that the range of $\tX(t)$ lies within the bounds
$$
-t\nu^{1/4}\leq x\leq (1-t)\nu^{1/4}.
$$

We fix  $h\in(0,1)$ and focus on  $t\in [0,1-h]$, so achieving uniformly in this range
\begin{equation}\label{psi-ord}
\tpsi(t,x)=O(\nu^{-1/4}),
\end{equation}
and for $k\geq 1$
\begin{equation}\label{psi-k}
\tpsi^k(t,x)=\left( 2-\frac{2x}{\nu^{1/4}(1-t)}\right)^{k/2}\nu^{-k/4}+O(\nu^{-(k+2)/4}),~{\rm for~}x\leq (1-t)\nu^{1/4} -\frac{1}{\nu^{3/4}(1-t)}
\end{equation}
as dictated by the bounds (\ref{small x}), (\ref{big x}).

Now let ${\cal D}$  be the space of  vanishing at infinity functions  $f\in C^3_0([0,1]\times {\mathbb R}^2)$ which satisfy a rapid decrease property 
$$ \sup|x^kf_{\bullet}(t,x,\ell)|<\infty,$$
where $f_{\bullet}$ is any derivative of $f$  of the first or second order and $k>0$.
Set
$$D^{>}_{h,\nu}:=\{(t,x,\ell): t\in[0,1-h],~|x|>\nu^{1/16}\},~~~D^{<}_{h,\nu}:=\{(t,x,\ell): t\in[0,1-h],~|x|\leq\nu^{1/16}\}.$$
We shall be using that for $f\in {\cal D}$ 
\begin{equation}\label{bigg-x}
\lim_{\nu\to\infty}\sup_{D^{>}_{h,\nu}}|\nu^k f_\bullet (x)|=0.
\end{equation}

The integrand in (\ref{genXL}) expands as
$$
f(t,x+u,\ell+v)-f(t,x,\ell)= f_x u+f_\ell v+ \frac{1}{2}f_{xx}\, u^2+f_{x\ell} \,uv+\frac{1}{2}f_{\ell\ell}\,v^2+R,
$$
where the remainder can be estimated as
$$ |R|\leq c \sum_{i=0}^3 u^iv^{3-i},$$ 
with constant $c$ chosen bigger than the maximum absolute value of any third derivative of $f$. Hence for the integrated remainder we have a uniform estimate
$$\nu^{3/4}\left| \int_0^{\tpsi} R du \right|\leq    \nu^{3/4}c \sum_{i=1}^4 \tpsi^i v^{4-i}=    O(\nu^{-1/4}),$$
using (\ref{psi-ord}), (\ref{vi}).

Integrating the Taylor polynomial yields 
$$
\widetilde{\cal L}_\nu f=f_t -\nu^{1/4}(f_x+f_\ell)+\nu^{3/4} \left\{ \frac{1}{2}f_x\, \tpsi^2+f_\ell \,v\tpsi + \frac{1}{6}\,f_{xx} \tpsi^3+\frac{1}{2}f_{x\ell}\,\tpsi^2 v+\frac{1}{2}f_{\ell\ell} v^2\tpsi\right\}+O(\nu^{-1/4}).
$$
Applying (\ref{bigg-x})
\begin{equation}\label{A-range}
\lim_{\nu\to\infty} \sup_{D^{>}_{h,\nu}}|\widetilde{\cal L}_\nu f(t,x,\ell)|=0.
\end{equation}
Thus we focus on the range $D^{<}_{h,\nu}$, where (\ref{small x}) and (\ref{psi-k}) can be employed. 
From (\ref{small x})
$$
-\nu^{-1/4}f_x+\nu^{3/4}\frac{1}{2}f_x\tpsi^2=-\frac{x}{1-t}f_x+O(\nu^{-1/4}).
$$
Observing that in this range  $|x\nu^{-1/4}|\leq \nu^{-3/16}$ for $k>0$ we expand as
  $$\tpsi^k(t,x,\ell)=2^{k/2}\nu^{-k/4} 
-\frac{2^{k/2-1}x}{1-t}
\nu^{-(k+1)/4}+O(\nu^{-(k+1)/4-1/8}),$$
with the remainder estimate being uniform over $D^{<}_{h,\nu}$. The remaining calculations  is  a careful  book-keeping using this formula and that the derivatives are uniformly bounded:

\begin{eqnarray*}
-\nu^{1/4}f_\ell +\nu^{3/4} f_\ell v\tpsi&=& -\frac{x}{2(1-t)}f_\ell+O(\nu^{-1/8}),\\
\nu^{3/4}\frac{1}{6}f_{xx}\tpsi^3&=&\frac{\sqrt{2}}{3}f_{xx}+O(\nu^{-3/16}),\\
\nu^{3/4} \frac{1}{2}f_{x\ell}\tpsi^2 v&=&   \frac{1}{\sqrt{2}}f_{x\ell}+ O(\nu^{-3/16}),\\
\nu^{3/4} \frac{1}{2}f_{\ell\ell}v^2\tpsi&=& \frac{1 }{2\sqrt{2}}f_{\ell\ell}+  O(\nu^{-3/16}).
\end{eqnarray*}

Define operator
\begin{equation*}
\widetilde{\cal L}f:=f_t -\frac{x}{1-t}f_x-\frac{x}{2(1-t)}f_\ell +  \frac{{\sigma_1}^2}{2} f_{xx} +\frac{{\sigma_2}^2}{2} f_{\ell\ell}+ \sigma_1 \sigma_2 \rho f_{x\ell},
\end{equation*}
with $\sigma_1,\sigma_2$, and $\rho$ given by (\ref{sigmas}).

\begin{lemma}\label{core-conv}   For $f\in {\cal D}$ and $h\in (0,1)$
$$
\lim_{\nu\to\infty} \sup_{(t,x,\ell)\in [0,1-h]\times{\mathbb R}^2} |\widetilde{\cal L}_\nu f(t,x,\ell)-\widetilde{\cal L}f(t,x,\ell)|=0
$$
\end{lemma}
\begin{proof}
The supremum over $D^{>}_{h,\nu}$ goes to zero since by (\ref{bigg-x}) the analogue of (\ref{A-range}) holds true for $\widetilde{\cal L}$. 
The supremum over $D^{<}_{h,\nu}$ goes to zero by the above expansions.

\end{proof}

Operator $\widetilde{\cal L}$ is the generator of a Gaussian diffusion process 
which satisfies   the stochastic differential equation
\begin{eqnarray}\label{OU}
\label{FirstSDE}
{\rm d}Y_1(t)&=& -\frac{Y_1(t)}{1-t}\,{\rm d}t\,+{\rm d}W_1(t),\\
\label{SecondSDE}
{\rm d}Y_2(t)&=& -\frac{Y_1(t)}{2(1-t)}\,{\rm d}t\,+{\rm d}W_2(t),
\end{eqnarray}
with zero initial value,
where ${\boldsymbol W}=(W_1,W_2)$ is the two-dimensional Brownian motion with covariance $\boldsymbol\Sigma$ introduced in 
(\ref{Sigma}).

From the equation for the first component (\ref{FirstSDE}), 
it is seen that $Y_1$  is a Brownian bridge 
\begin{equation}\label{BrBr}
Y_1(t)=(1-t)\int_0^t \frac{{\rm d}W_1(s)}{1-s},
\end{equation}
with the covariance  function ${\rm Cov}(Y_1(s),Y_1(t))= \sigma_1 s(1-t)$,   $0\leq s\leq t\leq 1.$ 
In particular, $Y_1(1)=0$. We shall discuss the second component later on.

The space   ${\cal D}$ is dense in a larger space $C^3_0([0,1-h]\times{\mathbb R}^2)$. 
Since the differentiability properties of functions are preserved under averaging over normally distributed translations,
$\cal D$ is invariant under the semigroup of $\boldsymbol Y$.
Thus by Watanabe's theorem (see \cite{Kallenberg}, Proposition 17.9) $\cal D$ is a core of  operator $\widetilde{\cal L}$. 
The above Lemma \ref{core-conv} and  Theorem 17.25  from \cite{Kallenberg} now imply
  weak convergence 
\begin{equation}\label{weak-h}
(\tX_\nu,\tL_\nu)\Rightarrow (Y_1,Y_2)~~{\rm in ~~}D[0,1-h]
\end{equation}
for every $h\in (0,1)$.    
A closer inspection of the above approximation errors suggests  that the quality of convergence deteriorates as $h\to 0$.

We encountered the  Brownian motion   ${\boldsymbol W}$  in connection with the free-endpoint stationary strategy in Sections \ref{stationary} and \ref{fstationary}.
Now we  see that  the variable  control (\ref{psi-appr})  causes a drift  that, in the $\nu\to\infty$ limit, forces  
the running maximum to timely arrive  at the north-east corner of the square.

\section{Convergence to diffusion: end of the proof}

The  martingale problem for $\widetilde{\cal L}$ is  well-posed on the complete interval, and the SDE (\ref{OU}) has a unique strong solution.
This suggests to extend convergence (\ref{weak-h}) to the full $[0,1]$. To that end, we need 
to monitor the behaviour of $\widetilde{\cal L}_\nu f$ for $t$ close to 1. 
Estimates in 
Bruss and Delbaen (\cite{BD2}, p. 294) show that 
 $\tX_\nu(1)\to 0$ in probability, which agrees neatly with the Brownian bridge limit,
but this still does not exclude giant fluctuations of the pre-limit process near $t=1$.

A similar kind of difficulty appears by the martingale approach to the classic problem of convergence of the empirical distribution function \cite{Hmaladze, JS}. 
The proof found in Jacod and Shiryaev (see \cite{JS}, p.561)
handles the nuisance
by  exploiting the time reversibility   of the Brownian bridge.
Our argument will rely  on the self-similarity.

Since (\ref{weak-h}) entails convergence of finite-dimensional distributions for times $t<1$ and ensures that the modulus of continuity behaves properly over $[0,1-h]$,
to justify tightness of $\tX_\nu$'s, and hence their convergence on $[0,1]$, it will be enough to show that
\begin{equation}\label{h-cond}
\lim_{h\to 0} \limsup_\nu
\prob(\sup_{t\in[1-h,1]}|\tX_\nu(t)|> h^{1/4})=0.
\end{equation}

Define $\xi_{\nu,h}$ by setting
$$\tX_\nu(1-h)=\sigma_1\sqrt{h(1-h)}\,\xi_{\nu,h}.$$
Since $\tX_\nu(1-h)\stackrel{d}{\to} Y_1(1-h)$ the distribution of $\xi_{\nu,h}$ is close to ${\cal N}(0,1)$ for large $\nu$.

By self-similarity of the selection strategy, $((X_\nu (t)-t),~t\in[1-h,1])$ has the same distribution as $(h^{-1}(X_{\nu h^2}(t)-t),~t\in [0,1])$ with the initial value $X_{\nu h^2}(0)= \nu^{-1/4}\sigma_1\sqrt{(1-h)/h}\,\xi_{\nu,h}$, as is seen
by zooming in  the  corner square north-east of the point $(1-h,1-h)$ with factor $h^{-1}$.  Changing variable $\nu h^2\to\nu$,   (\ref{h-cond}) translates as a compact containment condition
\begin{equation}\label{c-cont}
\lim_{h\to 0}\limsup_\nu \prob(\sup_{t\in[0,1]}|\tX_\nu(t)|>h^{-1/4})=0
\end{equation}
under the initial value $\tX_\nu (0)=\sqrt{1-h} \,\xi_{\nu,h}$.

To verify (\ref{c-cont}) we shall squeeze the running maximum $X$ between $X^\downarrow$ and $X^\uparrow$ whose normalised versions satisfy the compact containment 
condition.
We force the majorant and the minorant to live on the opposite sides of the diagonal. Both have  independent, almost stationary increments, so that functional limits 
can be readily identified.
For simplicity we will assume $X_\nu(0)=0$. The general case with  $X_\nu(0)$ of the order $\nu^{-1/4}$ can be handled by the same method.

\subsection{Majorant}\label{majorant}

Define process $X^\uparrow= X^\uparrow_\nu$ as solution to 

$${\rm d}X^\uparrow(t)=  \int_0^{\psi^\uparrow(t)}x{\Pi^*} ({\rm d}t{\rm d}x) + 1(X^\uparrow (t)=t){\rm d}t ,$$ 
$X^\uparrow(0)=K\nu^{-1/2}$ for some big enough $K>0$, with control
\begin{equation*}
\psi^\uparrow(t):=\sqrt{\frac{2}{\nu}}+ \frac{\beta}{\nu(1-t)}\,1(t\leq 1-K\nu^{-1/2})
\end{equation*}
not depending on $x$. Notation $1(\cdots)$ is used for indicators.
The process never drops below the line $x=K\nu^{-1/2}+t$, and whenever the line is hit the path drifts along it for some time.
By the construction, above the diagonal  the process $X^\uparrow$ increases faster than $X$, and is, in fact, a majorant.
\begin{lemma}\label{maj} By coupling via  {\rm (\ref{knaps})}, $  X^\uparrow \geq X$ a.s.
\end{lemma}

\begin{proof} 

By the  virtue of  (\ref{small x}),  (\ref{big x})  and definition of $\psi^\uparrow$ we have   $\psi^\uparrow(t,x)>\psi(t,x)$  for $x>t, ~  t\leq 1-K\nu^{-1/2}$.   
Hence by  Lemma \ref{compare}, ${\rm d}\{X^\uparrow( t)-X(t)\}>0 $ conditional on  $X^\uparrow(t)>X(t)>t$ at  time $t<1-K\nu^{-1/2}$.

Initially $X^\uparrow(0)>X(0)$, and $X^{\uparrow}(t)>1>X(t)$ for $t>1-K\nu^{-1/2}$.
Hence the only way the paths can cross is that $X$ overjumps $X^\uparrow$ from some position $x<t\leq X^\uparrow(t)$ at some time
$t\leq 1-K\nu^{-1/2}$. The latter  possibility is excluded, because
\begin{eqnarray*}
\psi(t,x)&<&\sqrt{\frac{2}{\nu}}\left(1+\frac{t-x}{1-t} \right) +\frac{\beta}{\nu(1-t)}< \sqrt{\frac{2}{\nu}} \left(1+\frac{t-x}{2(1-t)} \right) +\frac{\beta}{\nu(1-t)}\\ &\leq&
 \sqrt{\frac{2}{\nu}}+\frac{t-x}{K\sqrt{2}}+\frac{\beta}{K\sqrt{\nu}}<t-x+\frac{K}{\sqrt{\nu}}
\end{eqnarray*}
for $K$ chosen big enough.
\end{proof}

Let 
$$S(t):= \int_0^t\int_0^{\psi^\uparrow(t)}x \Pi^* ({\rm d}s{\rm d}x)-t.$$ 
This is a process with independent increments, which we can split into two independent components
$$S(t)=\left(\int_0^t \int_0^{\sqrt{2/\nu}}x\Pi^*({\rm d}s{\rm d}x)-t\right)+ \int_0^t\int_{\sqrt{2/\nu}}^{\psi^\uparrow(t)}x\Pi^*({\rm d}s{\rm d}x).$$
The mean value of the second part is estimated as
$$\frac{2 \nu }{\sqrt{\nu} }\int_0^{1-K/\sqrt{\nu}} \frac{\beta}{\nu(1-t)}{\rm d}t=O\left( \frac{\log \nu}{\sqrt{\nu}}\right),$$
and the first is a compensated compound Poisson process. Thus    $\nu^{1/4}S\Rightarrow W_1$ as $\nu\to\infty$.

Processes akin to $(X^\uparrow(t)-t, ~t\in[0,1])$ are common in applied probability \cite{Asmussen, Borovkov}. In particular, by the interpretation as the content of a single-server ${\rm M/G/1}$ queue,  the positive increments present jobs that
 arrive by Poisson process and are measured in terms of the demand on the service time.  The downward drift occurs due to the unit processing rate when the server is busy.
Borrowing a useful identity, 
$$X^\uparrow(t)-t=   S(t)-\inf_{u\in[0,t]} S(u),$$
we conclude on the weak convergence
 $(\nu^{1/4} (X^\uparrow(t)-t),~t\in [0,1])\Rightarrow |W_1|$ 
to a reflected  Brownian motion.

 \subsection{Minorant}\label{minorant}

This time we define $X^\downarrow$ by (\ref{knaps}) with  control
\begin{eqnarray*}
	\psi^\downarrow (t,x)=\begin{cases}    \left(\sqrt{\frac{2}{\nu}}-\frac{\beta}{\nu(1-t)}\right)  \wedge (t-x), ~~{\rm for~~}~0\leq t\leq 1-{K}/{\sqrt{\nu}},\\
		~~~~~0,~~~{\rm for~~}  ~~1-{K}/{\sqrt{\nu}}<  t\leq 1.
	\end{cases}  
\end{eqnarray*}
where $K$ is sufficiently large.
We can regard this as a suboptimal strategy that never selects marks $x>t$.
Starting at state $0$, the running maximum process stays below the diagonal throughout, and gets frozen at $t=1-K/\sqrt{\nu}$.
A counterpart of Lemma \ref{maj}, $X^\downarrow<X$ a.s., is readily checked. 

Switching general $\beta>0$ to $\beta=0$ 
impacts $\E X^\downarrow(t)$ by $O(\nu^{-1/2}\log\nu)$ uniformly in $t\in[0,1]$. Indeed,
the  jumps  are bounded by $2/\sqrt{\nu}$, and the expected number of jumps increases  by $O(\log \nu)$.

Assuming $\beta=0$, the process $(X^\downarrow(t)-t,~t\in[0,1-K\nu^{-1/2}])$ is a compensated compound Poisson process on the negative halfline, 
with reflection at $0$. We have therefore 
$$(\nu^{1/4} (X^\downarrow(t)-t),~t\in [0,1])\Rightarrow -|W_1|.$$
A rigorous proof can be obtained by inspecting convergence of the generator  acting on the functions $f\in{\cal D}$ with $f_x(t,0)=0.$

\subsection{The length process near termination}

Having established weak  convergence of $\tX$,  we wish to estimate fluctuations of $\tL$   near $t=1$. 
To that end, we aim to verify that
\begin{equation}\label{Lterm}
\lim_{h\to 0}\limsup_\nu 
\prob(\sup_{t\in[1-h,1]} |\tL(t)-\tL(1-h) |>\epsilon)=0.
\end{equation}

Write $s=1-h$ and split the difference in (\ref{Lterm}) in three parts
$$\tL(t)-\tL(s)= \nu^{1/4} P_1(t)- \nu^{1/4}P_2(t)+\nu^{1/4}P_3(t),$$ 
where
\begin{eqnarray*}
P_1(t)&:=&
(2\nu)^{-1/2}\{M(t)-M(s)\},\\
P_2(t)&:=& 
(2\nu)^{-1/2}F(\nu(1-t)(1-X(t)))-(1-t),\\
P_3(t)&:=&(2\nu)^{-1/2}F(\nu(1-s)(1-X(s)))-(1-s).
\end{eqnarray*}

From (\ref{exF}),
\begin{equation*}
\lim_{\nu\to\infty} \sup_{z\in [0,1]}\nu^{1/4} |(2\nu)^{-1}F(\nu z)-z|=0.
\end{equation*}
Using this,
definition of $\tX$ and that $|1-\sqrt{1-z}|\leq |z|$ for $z<1$ we obtain
\begin{align*}
\begin{split}
&|P_2(t)| \leq      | \sqrt{(1-t)(1-X(t))}-(1-t)|      +      \\
& \{(2\nu)^{-1/2}F(\nu(1-t)(1-X(t)))-\sqrt{(1-t)(1-X(t))}  \}   \leq \\
 &\left|(1-t)\left(\sqrt{1-\frac{\tX(t)}{\nu^{1/4}(1-t)}}-1\right)\right|+ \sup_{z\in [0,1]} |(2\nu)^{-1}F(\nu z)-z|\leq \\
&\nu^{-1/4}{|\tX(t)|}+\sup_{z\in [0,1]} |(2\nu)^{-1}F(\nu z)-z|
=\nu^{-1/4}{|\tX(t)|}+o(\nu^{-1/4}),
\end{split}
\end{align*}
so from (\ref{c-cont})
\begin{equation}\label{Spart}
 \lim_{h\to 0}\limsup_\nu 
\prob(\sup_{t\in[1-h,1]} \nu^{1/4}|P_2(t)|>\varepsilon/3)=0.
\end{equation}
This  relation also holds for $P_3$.

For the first part, apply Doob's maximal inequality
\begin{equation}\label{Fpart}
\prob\left(\sup_{t\in[1-h,1]} \nu^{1/4}|P_1(t)|>\varepsilon/3\right)\leq \frac{9}{2 \varepsilon^{2}\sqrt{\nu}} \,\,{\rm Var}\{ M(1)-M(1-h)\}.
\end{equation}
In terms of  the quadratic variation  (see  \cite{Bremaud}, Chapter 2)
$$
{\rm Var}\{ M(1)-M(s)\}=\me\int_s^{1} \nu \,\psi(t,X(t))\,\varphi(t,X(t)){\rm d}t,
$$
where
$$\varphi(t,x)=\me \{1+F(\nu(1-t)(1-x-U\psi(t,x)))-F(\nu(1-t)(1-x))\}^2$$
(with $U$ uniform on $[0,1]$) is the mean-square size of the generic jump of $M$.
Under the optimal strategy $0\leq\varphi(t,x)\leq 1$ (finer estimates are in \cite{BD2}, Section 4), and from (\ref{exF})  
and (\ref{delta-asymp}) we have a 
 uniform bound $|\varphi(t,x)|<c$.
Whence
$$
{\rm Var}\{ M(1)-M(1-h)\}<c \,\me\int_{1-h}^{1} \nu \,\psi(t,X(t)){\rm d}t= c \,\me\,\{L(1)-L(1-h)\}<c \sqrt{2\nu h},
$$
the probability in (\ref{Fpart}) is estimated as $O(\sqrt{h})$, and (\ref{Lterm}) follows from this and (\ref{Spart}).

\section{Main result}

By the domination argument, tightness of $(\tX_\nu, \tL_\nu)$ follows on the whole $[0,1]$, and we arrive at our main result.
\begin{thm}\label{main res} The normalised running maximum and the length process {\rm(\ref{scale})} driven by a control satisfying {\rm (\ref{ss-psi})} and {\rm (\ref{delta-asymp})} 
(in particular, under the optimal online selection strategy)
converge weakly in the Skorokhod space $D[0,1]$,
\begin{equation*}
(\tX_\nu,\tL_\nu)\Rightarrow (Y_1,Y_2), ~~~{\rm as~~}\nu\to\infty,
\end{equation*}
 where the limit  bivariate process is a Gaussian diffusion defined by the equations {\rm (\ref{OU}), (\ref{SecondSDE})} with zero initial conditions.
\end{thm}

We observed already that $Y_1$ is the Brownian bridge (\ref{BrBr}) and from (\ref{SecondSDE})
$$Y_2(t)=\frac{Y_1(t)}{2}-\frac{W_1(t)}{2}+W_2(t),$$
so splitting the martingale part in independent components, we get, explicitly,
\begin{equation}\label{Y2form}
Y_2(t)= \int_0^t \frac{(1-t)}{2(1-s)}{\rm d}W_1(s)+ \frac{1}{4}\,W_1(t)+ \left(W_2(t)-\frac{3}{4}\,W_1(t)  \right),
\end{equation}
which is a sum of a Brownian motion, derived Brownian bridge,   and another independent Brownian motion.

To  find the covariance structure, it is convenient to resort to matrix calculations. 
We may write the solution to (\ref{FirstSDE}), (\ref{SecondSDE}) as
$$\boldsymbol{Y}(t)^T =  e^{a(t)}\int_0^t e^{-a(u)}{\rm d}{\boldsymbol W}(u)^T,$$    
where
$$a(t):=A\int_0^t  \frac{1}{1-u}{\rm d}u   =A \log{(1-t)}, ~~~~~ A:= \begin{pmatrix}1&0\\ \frac{1}{2} &0 \end{pmatrix},$$
which yields by the It{\'o} isometry
$$\me\{\boldsymbol{Y}(s)^T\boldsymbol{Y}(t)\}=\int_0^t e^{a(s)-a(u)}{\boldsymbol\Sigma}e^{(a(t)-a(u))^T}{\rm d}u,~~~~~0\leq s\leq t\leq 1.$$
Since $A$ is an idempotent matrix, the exponents are readily calculated as
\begin{eqnarray*}
e^{a(t)} &=& \sum_{i=0}^\infty \frac{A^i (\log{(1-t)})^i}{i!} = I + A\sum_{i=1}^\infty \frac{(\log{(1-t)})^i}{i!} = I- tA = 
\large{
\begin{pmatrix}[1.2]
1-t & 0 \\
-\frac{t}{2} & 1
\end{pmatrix}
},\\
e^{-a(t)}&=& 1+\frac{t}{1-t}A=
\large{
\begin{pmatrix}[1.2]
\frac{1}{1-t} & 0 \\
\frac{t}{2(1-t)} & 1
\end{pmatrix}.
}
\end{eqnarray*}
With a minor help of {\tt Mathematica} we arrive
at the cross-covariance matrix
\begin{eqnarray*}
\mathbb{E}\{\boldsymbol{Y}(s)^T\boldsymbol{Y}(t)\} &=&
{\large
\begin{pmatrix}[1.8]
\frac{2\sqrt{2} \, s(1-t)}{3} & \frac{2s(1-t)-(1-s)\log{(1-s)}}{3\sqrt{2}} \\
\frac{(1-t) (2s-\log{(1-s)})}{3\sqrt{2}}
 & \frac{2s(2-t)-(2-s-t)\log{(1-s)})}{6\sqrt{2}}
\end{pmatrix},
}
\end{eqnarray*} 
where $0\leq s\leq t\leq 1$.

The following graphs illustrate the covariance structure of ${\boldsymbol Y}(t)$.
\begin{figure}[h!]
	\centering
	\begin{subfigure}[b]{0.45\linewidth}
		\includegraphics[width=\linewidth]
{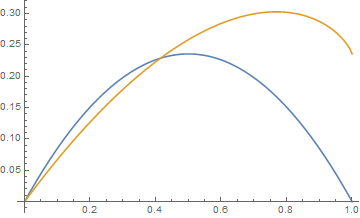}

		\caption{${\color{YellowOrange} {\rm Var} Y_1(t)}, {\color{MidnightBlue} {\rm Var} Y_2(t)}$}
	\end{subfigure}
	\begin{subfigure}[b]{0.45\linewidth}
		\includegraphics[width =\linewidth]{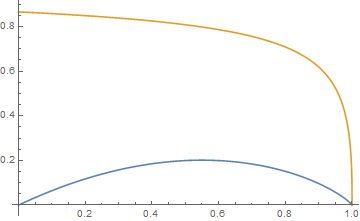}
		\caption{${\color{MidnightBlue} {\rm Cov}(Y_1(t), Y_2(t))},   {\color{YellowOrange} {\rm Corr}(Y_1(t), Y_2(t))}  $}
	\end{subfigure}
\end{figure}

The limit length process $Y_2$ is {\it not} Markovian, since its covariance function does not satisfy the factorisation criterion (see \cite{MarcusRosen}, p. 148).
The sum of  two first terms in (\ref{Y2form}) is non-Markovian too.

\section{Derived processes}

From Theorem \ref{main res} follow  functional limits for 
 normalised compensators and martingale (\ref{Mmart}) 
\begin{eqnarray*}
\nu^{1/4}(C_X(t)-t, ~t\in [0,1])&\Rightarrow &Y_1-W_1,\\
2\nu^{1/4}\left(\frac{C_L(t)}{\sqrt{2\nu}}-t, ~t\in [0,1]\right)&\Rightarrow&  Y_1-W_1,\\
\sqrt{2}\nu^{-1/4}M &\Rightarrow & 2W_2-W_1 \stackrel{d}{=} W_1,
\end{eqnarray*}
 with account of (\ref{IdentW}).
A counterpart of  (\ref{2L-X}) becomes
$$ \tL-2\tX\Rightarrow W_1.$$
Notably, the limit distributions for $t=1$ are all the same ${\cal N}(\sigma_1^2,0)$.

For a  normalised square-root process
$$\widetilde{Z}(t):= \nu^{1/4}\left( \frac{Z(t)}{\sqrt{2\nu}}-(1-t) \right) $$
we have  $\widetilde{Z}\Rightarrow -\frac{1}{2}Y_1$.
In \cite{GS} we showed that the range of $Z$ at big distance  from $0$ can be split in almost independent renewal cycles with distribution close to that of $(E/2+U)/\sqrt{2}$,
where $E$ and $U$ are independent standard exponential and uniform variables.

From these limit relations the result of \cite{BD2} on the  joint convergence of normalised compensated $X$ and $L$ to $\boldsymbol W$ easily follows.
Bruss and Delbaen also proved the Brownian limit for the martingale $M$, which by virtue of $M(1)=L(1)-F(\nu)$ lead them to the central limit theorem for the total length $L(1)$.

It is of interest to look at the  distributions of  the pairs $(X(t),C_X(t))$ and $(L(t),C_L(t))$,
to capture dependence between the processes and their compensators.
In the  $\nu\to\infty$ limit  these approach the bivariate normal distributions of 
 $(Y_1(t), Y_1(t)-W_1(t))$ and $(Y_2(t), \frac{1}{2}(Y_1(t)-W_1(t)))$, respectively.
Calculation of the covariance matrices is straightforward from our previous findings complemented by the formula
$${\rm Cov}(Y_1(t),W_1(t))= -\sigma_1^2(1-t)\log(1-t)$$
obtained by the It{\'o} isometry. For instance
\begin{eqnarray*}
{\rm Var}\left\{Y_1(t)-W_1(t) \right\}=\frac{4\sqrt{2}}{3} \left(t - \frac{t^2}{2} + (1 - t) \log{(1-t)}\right).\\
\end{eqnarray*}

\begin{figure}[h!]
	\centering
	\begin{subfigure}[b]{0.45\linewidth}
		\includegraphics[width=\linewidth]
{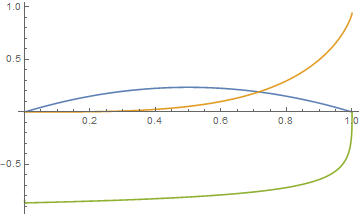}
\caption{
${\color{YellowOrange} {\rm Var} \, Y_1(t)}, {\color{MidnightBlue} {\rm Var}\{ Y_1(t)-W_1(t)\}},~~~~~~~$ {\color{OliveGreen} $~~~~~~{\rm Corr}\{Y_1(t),Y_1(t)-W_1(t)\}$.} }
	\end{subfigure}
	\begin{subfigure}[b]{0.45\linewidth}
		\includegraphics[width =\linewidth]{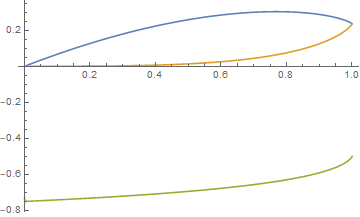}
\caption{
${\color{YellowOrange} {\rm Var} \, Y_2(t)}, {\color{MidnightBlue} {\rm Var}\{\frac{1}{2}( Y_1(t)-W_1(t)\}},~~~~~~~$, {\color{OliveGreen} $~~~~~~{\rm Corr}\{Y_2(t),\frac{1}{2}(Y_1(t)-W_1(t))\}$.} }
	\end{subfigure}
\end{figure}

\section{Convergence of the moments}\label{Moments}

The weak convergence shown in 
Theorem \ref{main res}, combined with the convergence of  moments for the majorant and minorant processes imply by virtue of `Pratt's lemma'  the expansion
$$\E (X(t)-t)^k=\nu^{-k/4}\,\E \, Y_1^k(t)+o(\nu^{-k/4}), ~~~k\in{\mathbb N},$$ 
along with a similar expansion for the $t$-centred moment functions of $ L/\sqrt{2\nu}$. For $k=1$  the leading term vanishes, hence  the 
convergence rate should be higher, as is evidenced  in the instance $t=1$ by  (\ref{q1asymp}).
The logarithmic factor in (\ref{q1asymp}) results from the optimality gap, hence it is of interest to inspect how the gap emerges in the course of selection.

We choose smallest possible constants 
 $\beta_->0, \beta_+\geq 0$ to squeeze the control function in the bounds
\begin{equation}\label{psi UB}
\sqrt{\frac{2(1-x)}{\nu(1-t)}} - \frac{\beta_-}{\nu (1-t)} \leq \psi(t,x) \leq \sqrt{\frac{2(1-x)}{\nu(1-t)}} + \frac{\beta_+}{\nu (1-t)}, ~~~ t,x\in[0,1).
\end{equation}
The condition  (\ref{small x}) thus holds  with $\beta\geq \max(\beta_-, \beta_+,1)$.
To motivate introducing two parameters we note  that  
for the optimal strategy (\ref{psi UB}) holds with $\beta_+=0$ 
(\cite{BD2}, Equation (3.5)), and that there is some asymmetry 
in the upper and lower estimates below.

The following auxiliary result is a special case of Gr{\"o}nwall inequality:

\begin{lemma}\label{LL} Suppose function $f$ with $f(0)=0$ satisfies the integral inequality
$$
f(t)\leq -\int_0^t f(s)\left(  \frac{1}{1-s}+\frac{a}{(1-s)^2}\right)  {\rm d}s+\int_0^t g(s){\rm d}s, ~~~t\in [0,1),
$$
$a\in{\mathbb R}$.
Then 
\begin{equation}\label{fsol}
f(t)\leq (1-t)e^{-\frac{a}{1-t}}    \int_0^t \frac{e^{\frac{a}{1-s}}\,g(s)}{1-s}{\rm d}s,  ~~~t\in [0,1).
\end{equation}
\end{lemma}
\begin{proof}
The linear operator defined by 
the right-hand side of (\ref{fsol}) gives a solution to the associated integral equation. 
The assertion follows by observing that nonnegative $g$ is mapped to nonnegative $f$.
\end{proof}

\subsection{Bounds on $p(t)$}

The upper bound in  (\ref{psi UB}) implies
\begin{eqnarray}\label{psi-up}
\frac{\nu}{2}\psi^2(s,x) -1\leq -\frac{x-s}{1-s}+\frac{\beta_+\sqrt{2}}{\sqrt{\nu}(1-s)}\sqrt{\frac{1-x}{1-s}}+\frac{\beta_+^2}{2\nu(1-s)^2}.
\end{eqnarray}
Using the elementary inequality  $\sqrt{1-z}\leq1-z/2$ for $z\leq1$   we obtain
$$
\E \sqrt{\frac{1-X(t)}{1-t}}\leq 1- \frac{\E X(t)-t }{2(1-t)}\leq 1-\frac{p(t)-t}{2(1-t)},
$$
and integrating in (\ref{psi-up}) yields 
$$
p(t)-t\leq -\int_0^t (p(s)-s) \left( \frac{1}{1-s}+ \frac{b}{(1-s)^2}\right)  {\rm d}s +                             \int_0^t g(s){\rm d}s,
$$
where
\begin{eqnarray*}
g(t)=\frac{2b  }{(1-t)} + \frac{b^2}{(1-t)^2}, ~~~~b=\frac{\beta_+}{\sqrt{2\nu}}.
\end{eqnarray*}
Applying Lemma \ref{LL} with $f(t)=p(t)-t$ and $a=b$ we obtain
$p(t)-t\leq G(b,t),$
where
$$G(b,t):=1+b-t-(1+b)(1-t)\exp\left(-\frac{bt}{1-t}\right).$$

For small $b>0$, this is a concave function, with $G(b,t)-bt$ changing sign from $+$ to $-$ at some point  approaching $2/3$ as  $b\to 0$.
The asymptotic expansion 
$$G(b,t)\sim 2tb+ \frac{(2t-3t^2)b^2}{2(1-t)}, ~~~b\to0,$$
holds uniformly, at least for $t$ bounded away from $1$; 
therefore there is  an upper bound $G(b,t)<2bt+c_+b^2$, where
the  constant should be chosen to satisfy 
$$c_+>\max_{t\in[0,1]} \,\frac{2t-3t^2}{2(1-t)}= 2-\sqrt{3}.$$
It follows that
\begin{equation}\label{p-up}
p(t)-t\leq \frac{\beta_+\sqrt{2}\,\,t}{\sqrt{\nu}} + \frac{c_+\,\beta_+^2}{2\nu},
\end{equation}
 uniformly in $t\in[0,1]$ for  sufficiently big $\nu$.

To estimate in the opposite direction, we have  from the lemma $p(t)-t\geq G(b,t)$, this time with negative parameter
$$b=-  \frac{\beta_-}{\sqrt{2\nu}}.$$
Changing the variable to $T=(1-t)^{-1}$ simplifies analysis, and it is readily checked that 
$$T\mapsto T\, {\bigg(}G\left(b,1-T^{-1}\right)-2b\left(1-T^{-1}\right)-b^2T{\bigg)}$$
is a concave function, positive in the range $1\leq T<T_0$, where $T_0$ is such that 
$-b\, T_0$ approaches, as $|b|\to0$,   a limit  value $1.7933\ldots$  (the positive root of $1+x+x^2=e^x$), which we replace  by smaller $\sqrt{2}$. Thus 
\begin{equation}\label{p-down2}
p(t)-t\geq -\frac{\beta_-\sqrt{2}\,\,t}{\sqrt{\nu}} -   \frac{\beta_-^2}{2\nu(1-t)},~~{\rm for}~~t\leq1-\frac{\,\beta_-}{2\sqrt{\nu}},
\end{equation}
provided  $\nu$  is sufficiently large. But then by monotonicity from (\ref{p-down2}) it follows that
\begin{eqnarray}\label{extreme-p}
p(t)\geq p\left(1-\frac{\,\beta_-}{2\sqrt{\nu}}\right)\geq
1-\frac{\left(\sqrt{2}+\frac{3}{2}\right)\beta_-}{\sqrt{\nu}}, ~~~{\rm for ~~}t>1-\frac{\,\beta_-}{2\sqrt{\nu}},
\end{eqnarray}
hence in this range of $t$
\begin{eqnarray}\label{extreme-p}
p(t)-t&\geq& p(t)-1>-\frac{\left(\sqrt{2}+\frac{3}{2}\right)\beta_-}{\sqrt{\nu}},
\end{eqnarray}
(also note that the trivial upper bound $p(t)-t<1-t<\frac{\beta_-}{2\sqrt{\nu}}$ might improve upon (\ref{p-up}) in this range).

Bounding the second term in (\ref{p-down2}) by its maximum, 
and combining with (\ref{extreme-p}) into  single inequality we obtain an
 estimate with simpler constant $3>\sqrt{2}+3/2$
\begin{equation}\label{p-down3}
p(t)-t\geq  -\frac{3\beta_-}{\sqrt{\nu}}, ~~~t\in[0,1].
\end{equation}
Similarly, the second term in (\ref{p-up}) can be absorbed into the first with a larger constant.
With the full range $t\in[0,1]$  covered, we have shown that 
$$\sup_{t\in[0,1]}|p(t)-t|=O\left(\frac{1}{\sqrt{\nu}}\right).$$

\subsection{Bounds on $q(t)$}

We turn to $q(t)=\E L(t)/\sqrt{2\nu}$. For the upper estimate we use (\ref{pq}) to obtain an integral inequality
\begin{eqnarray*}
q(t)=\E\int_0^t \sqrt{\nu/2}\,\,\psi(s,X(s))ds\leq \E\int_0^t \left(\sqrt{\frac{1-X(s)}{1-s}} +\frac{\beta_+}{\sqrt{2\nu}(1-s)}\right){\rm d}s\leq\\
\int_0^t \left(1- \frac{p(s)-s}{2(1-s)} +\frac{\beta_+}{\sqrt{2\nu}(1-s)}\right){\rm d}s\leq 
\int_0^t \left(1- \frac{q(s)-s}{(1-s)} +\frac{\beta_+}{\sqrt{2\nu}(1-s)}\right){\rm d}s,
\end{eqnarray*}
then  apply Lemma \ref{LL} with $a=0$ to get 
\begin{equation}\label{q-up}
q(t)-t\leq  \frac{\beta_+\,\,t}{\sqrt{2\nu}}.
\end{equation}
The estimate approaches zero faster than in (\ref{q1asymp}), but there is no disagreement since $q(1)<1$.
Note that applying (\ref{pq}) and (\ref{p-up}) straight incurs a second term.

For the optimal strategy, (\ref{psi UB}) holds with $\beta_+=0$, thus in this case $p(t)-t\leq 0$ and $q(t)-t\leq 0$.

Obtaining the lower bound is more challenging.
Under the optimal strategy, 
the value function $F$ in (\ref{Mmart1}) is concave \cite{BD1}, but under our more general assumptions on $\psi$ this need not be the case. 
However,  by virtue of (\ref{exF}) we  may replace $F$  by the concave function 
\begin{equation}\label{Fhat}
\widehat{F}(\nu):=\sqrt{2\nu}-\frac{1}{12}\log (\nu+1),
\end{equation}
to obtain an expansion 
\begin{equation}\label{Lexp}
\E L(t)=\widehat{F}(\nu)- \E \{\widehat{F}(\nu(1-X(t))(1-t))\}+O(1),
\end{equation}
where the absolute value of the remainder  is bounded uniformly  in $t$ and $\nu$ by some constant $K$ only depending on $\beta_-$ and $\beta_+$.

By  monotonicity and concavity of $\widehat{F}$, using Jensen inequality and (\ref{p-down3}) we estimate
\begin{eqnarray*}
\E\{ \widehat{F}(\nu(1-X(t))(1-t))\}\leq
\widehat{F}\left(\nu(1-t) (1-p(t)\right)=~~~~~~~~~~~~\\
\widehat{F}\left(\nu(1-t)^2 \left(1-\frac{p(t)-t}{1-t}\right)\right)\leq~~~~~~~~~~~~\\
\widehat{F}\left(\nu(1-t)^2 \left(1 +     \frac{3\beta_-}{\sqrt{\nu}(1-t)}\right)\right)=~~~~~~~~~~~~\\
\sqrt{2\nu}(1-t) \sqrt{1 +     \frac{3\beta_-}{\sqrt{\nu}(1-t)}}
-\frac{1}{12}
\log\left\{\nu(1-t)^2 \left(1 +     \frac{3\beta_-}{\sqrt{\nu}(1-t)}\right) +1\right\}<\\
\sqrt{2\nu}(1-t) \left(1 +     \frac{3\beta_-}{2\sqrt{\nu}(1-t)}\right)
-\frac{1}{12}
\log\left(\nu(1-t)^2) \right)<\\
\sqrt{2\nu}(1-t)+\frac{3\beta_-}{\sqrt{2}} 
-\frac{1}{12}\log\nu -\frac{1}{6}\log(1-t).
\end{eqnarray*}
Substituting  this   along with   (\ref{Fhat})
into (\ref{Lexp}) we see that, for large enough $\nu$,
\begin{equation}\label{Ldown}
\E L(t)\geq \sqrt{2\nu}\,t+\frac{1}{6}\log(1-t)-   \left(\frac{3\beta_-}{\sqrt{2}} +K\right), ~~~~t\in[0,1].
\end{equation}

The logarithmic term makes  (\ref{Ldown}) useless 
for $t$ too close to $1$. However, cutting the range at, say $t_0:=1-1/\sqrt{\nu}$, we can just employ the monotonicity
to squeeze the expected length as
$$F(\nu)\geq \E L(t)\geq \E L(t_0)\geq \sqrt{2\nu}+\frac{1}{12}\log \nu -\left(\sqrt{2}+\frac{3\beta_-}{\sqrt{2}} +K\right),~~~t\geq t_0.$$

For a better overview, we re-write (\ref{q-up}) as 
\begin{equation}\label{q-up11}
\E L(t)\leq \sqrt{2 \nu}  \,t+\beta_+\,\,t.
\end{equation}
Comparing (\ref{Ldown}) with (\ref{q-up11}) it is seen that, uniformly in $t\in[0,1-h]$,  the mean selected length $\E L(t)$ is within $O(1)$ from $\sqrt{2\nu}\,t$, the latter being the exact mean length under the (unfeasible) stationary strategy.

With some more work we could show an upper bound with two leading terms as in (\ref{Ldown}) and a remainder uniformly bounded over $t<1-\nu^{-1/4+\epsilon}$.

\end{document}